\documentclass[a4paper]{article}

\usepackage{amsmath,amsthm,amssymb}
\usepackage{cancel}
\usepackage[left=2.9cm, right=2.9cm, top=2.9cm, bottom=2.9cm]{geometry}
\usepackage{hyperref}
\usepackage{mathtools}
\usepackage{subfig}
\usepackage{xcolor}
\allowdisplaybreaks

\definecolor{DarkPurple}{HTML}{381d2a}
\definecolor{QueenBlue}{HTML}{3e6990}
\definecolor{LaurelGreen}{HTML}{aabd8c}
\definecolor{DutchWhite}{HTML}{e9e3b4}
\definecolor{PinkOrange}{HTML}{f39b6d}
\definecolor{green}{rgb}{0.0, 0.5, 0.0}

\hypersetup{
citebordercolor = LaurelGreen,
urlbordercolor = QueenBlue,
linkbordercolor = PinkOrange
}

\newtheorem{theorem}{Theorem}[section]
\newtheorem{lemma}[theorem]{Lemma}
\newtheorem{remark}[theorem]{Remark}
\newtheorem{example}[theorem]{Example}

\newcommand{\calA}{\mathcal{A}}
\newcommand{\calB}{\mathcal{B}}

\newcommand{\calF}{\mathcal{F}}

\newcommand{\calV}{\mathcal{V}}

\newcommand{\bbC}{\mathbb{C}}

\newcommand{\bbR}{\mathbb{R}}
\newcommand{\bbZ}{\mathbb{Z}}

\newcommand{\rme}{\mathrm{e}}
\newcommand{\rmi}{\mathrm{i}}

\newcommand{\set}[2]{\left\{ #1 \,\middle|\, #2 \right\}}
\newcommand{\abs}[1]{\left\lvert #1 \right\rvert}
\newcommand{\norm}[1]{\left\lVert #1 \right\rVert}

\DeclareMathOperator{\supp}{supp}

\DeclareMathOperator{\sinc}{sinc}

\DeclareMathOperator{\Exists}{\exists}
\DeclareMathOperator{\Forall}{\forall}

\newcommand{\dd}{\,\mathrm{d}}

\title{Uniqueness of STFT phase retrieval for bandlimited functions}
\author{Rima Alaifari \& Matthias Wellershoff \\ Seminar for Applied Mathematics, ETH Z{\"u}rich.}
\date{\today}

\begin{document}
\maketitle

\begin{abstract}
    \noindent
    We consider the problem of phase retrieval from magnitudes of short-time Fourier transform
    (STFT) measurements. It is well-known that signals are uniquely determined (up to global phase)
    by their STFT magnitude when the underlying window has an ambiguity function that is nowhere
    vanishing. It is less clear, however, what can be said in terms of unique phase-retrievability
    when the ambiguity function of the underlying window vanishes on some of the time-frequency
    plane. In this short note, we demonstrate that by considering signals in Paley--Wiener spaces,
    it is possible to prove new uniqueness results for STFT phase retrieval. Among those, we
    establish a first uniqueness theorem for STFT phase retrieval \emph{from magnitude-only 
    samples} in a real-valued setting.
\end{abstract}

\section{Introduction}
\label{sec:introduction}

The problem of phase retrieval has been around since the very early days of X-ray crystallography
\cite{Bragg60,Sayre02}. To date, its applications include coherent diffraction imaging, astronomy
and audio processing. The measurements in phase retrieval problems typically consist of phaseless
Fourier-type data of the object of interest. Acquisition of magnitude-only measurements means loss
of information that needs to be accounted for. One possible approach for phase retrieval is to
collect redundant measurements as is done in ptychography for coherent diffraction imaging
\cite{Chapman96,Hegerl70,Rodenburg08}. There, the idea is that instead of creating one set of
measurements through diffraction, a sliding pinhole is added and many masked diffraction patterns
are collected. Hence, the measurements can be thought of as magnitudes of windowed Fourier
transforms. The same is true for measurements collected in audio processing such as the phase
vocoder \cite{Griffin84,Prusa17}. Indeed, suppose one wants to alter an audio signal (for instance 
pitch-shift it). To do so, one can take its short-time Fourier transform (STFT), redistribute the
magnitudes thereof in the time-frequency plane and look for a matching audio signal by performing phase
retrieval. Motivated by these applications, we consider phaseless measurements of short-time
Fourier transforms in this note. More precisely, we analyse the question of unique phase
retrievability from STFT magnitudes when the underlying signal is bandlimited.

\vspace{5pt}
\noindent
In general, rather little is known about the uniqueness of phase retrieval from STFT magnitude
measurements. A known result is that one may recover signals up to global phase from 
phaseless STFT measurements when the ambiguity function of the underlying window function is 
nowhere vanishing \cite{Groechenig01, Grohs18} (see Lemma \ref{lem:classical}). Additionally, the
\emph{complement property} \cite{Balan06,Bandeira14,Cahill16} is a necessary condition for
uniqueness. Finally, in the finite-dimensional setting, there is a plethora of
results for phase retrieval from discrete short-time Fourier magnitudes \cite{Eldar14, Pfander15}.
The aim of this note is to provide milder assumptions on the ambiguity function of the window that
guarantee uniqueness of STFT phase retrieval when the considered signal class is that of
bandlimited functions.

\vspace{5pt}
\noindent
Let us start by summarising relevant prerequisites and fixing notations: For a function $f \in
L^2(\bbR)$, we define the \emph{short-time Fourier transform (STFT)} (with window function $\phi
\in L^2(\bbR)$) by
\[
    \calV_\phi f(x,\omega) := \int_\bbR f(t) \overline{\phi(t-x)} \rme^{-2\pi\rmi t \omega} \dd t,
    \qquad x,\omega \in \bbR.
\]
It can be shown that $\calV_\phi f$ is uniformly continuous \cite{Groechenig01}. We consider the
phase retrieval problem of recovering $f$ from STFT magnitude measurements $\abs{\calV_\phi f}$.
Note that it is impossible to distinguish between $f$ and $\rme^{\rmi \alpha} f$, for $\alpha \in
\bbR$, from the STFT magnitude measurements alone. For this reason, we aim to reconstruct $f$
\emph{up to global phase}. That is, we attempt to recover the equivalence class
\[
    [f] := \set{ f \rme^{\rmi \alpha} }{ \alpha \in \bbR }.
\]

\vspace{5pt}
\noindent
One of the most important properties of the phase retrieval problem with STFT measurements is the 
ambiguity function relation. We use the convention 
\[
    \calF f(\xi) = \int_{\bbR} f(t) \rme^{-2\pi\rmi t  \xi} \dd t,
    \qquad \xi \in \bbR,
\]
for the Fourier transform on $L^1(\bbR)$ and extend it to $L^2(\bbR)$ by a density 
argument. In addition, we can define the \emph{ambiguity function} of a signal $f \in
L^2(\bbR)$ via 
\[
    \calA f(x,\omega) :=  \rme^{\pi \rmi x \omega} \calV_f f (x,\omega).
\]
The ambiguity function relation can now be stated as follows:

\begin{lemma}[Ambiguity function relation]\label{lem:amb}
Let $f,\phi \in L^2(\bbR)$. Then, 
\[
    \calF\left( \abs{\calV_\phi f }^2 \right)(\omega, -x)
    = \calA f (x,\omega) \cdot \overline{\calA \phi (x,\omega)},
    \qquad x,\omega \in \bbR.
\]
\end{lemma}

\noindent
We included a proof of this well-known relation in appendix \ref{app:ambiguity}, for the
convenience of the reader. One direct corollary of the ambiguity function relation is that if the 
ambiguity function of $\phi \in L^2(\bbR)$ is nowhere vanishing, then one may recover the 
ambiguity function of $f \in L^2(\bbR)$ everywhere from the STFT magnitude measurements
$\abs{\calV_\phi f}$. Furthermore, $f$ is uniquely determined up to global phase by its ambiguity 
function. Combining these observations, one has (see e.g.~\cite{Groechenig01,Grohs18}):

\begin{lemma}
\label{lem:classical}
Let $\phi \in L^2(\bbR)$ be such that 
\[
    \calA \phi (x,\omega) \neq 0, \qquad \mbox{for a.e.~} (x,\omega) \in \bbR^2.
\]
Then, the following are equivalent for $f,g \in L^2(\bbR)$:
    \begin{enumerate}
        \item $f = \rme^{\rmi\alpha} g$, for some $\alpha \in \bbR$.
        \item $\abs{\calV_\phi f } = \abs{\calV_\phi g}$.
    \end{enumerate}
\end{lemma}

\noindent
Clearly, if $\phi \in L^2(\bbR)$ is such that $\calA \phi$ is zero in some region of $\mathbb{C}$,
then, using the ambiguity function relation, one cannot recover $\calA f$ everywhere. In this note,
we ask whether under some additional assumptions, this scenario still enjoys unique phase recovery.
In particular, we consider bandlimited functions $f$. It turns out that in this setting it suffices
to assume that $\calA \phi$ does not vanish on certain line segments in the time-frequency plane.
More precisely, for $B>0$ and $p \in \{1,2\}$, we consider the \emph{Paley--Wiener space} of
bandlimited functions defined as
\[
    \mathrm{PW}^p_B := \set{f : \bbC \to \bbC}{\Exists F \in L^p([-B,B]) \Forall z \in \bbC :
        f(z) = \int_{-B}^B F(\xi) \rme^{2\pi\rmi \xi z} \dd \xi }.
\]
We record the following classical results on functions in the Paley--Wiener space that we will make
use of:

\begin{theorem}[Paley--Wiener theorem]\label{thm:pw}
Let $B > 0$. Then, the following are equivalent:
\begin{enumerate}
    \item $f \in \mathrm{PW}^2_B$.
    \item $f$ is an entire function such that there exists a constant $c > 0$ for which
    \[
        \abs{f(z)} \leq c \, \rme^{2 \pi B \abs{z}}, \qquad z \in \bbC,
    \]
    and 
    \[
        \int_\bbR \abs{f(t)}^2 \dd t < \infty.
    \]
\end{enumerate}
\end{theorem}

\begin{theorem}[WSK sampling theorem]\label{thm:wsk}
    Let $B > 0$ and $f \in \mathrm{PW}^2_B$. Then, we have 
    \[
        f(t) = \sum_{n \in \bbZ} f\left( \frac{n}{2B} \right) \sinc\left( 2B t - n \right),
        \qquad t \in \bbR.
    \]
\end{theorem}

\begin{theorem}[see Theorem 1 in \cite{Thakur10}, p.~723]
    \label{thm:tst}
    Let $p \in \{1,2\}$, let $B > 0$ and let $f \in \mathrm{PW}_B^p$ be real-valued on the real
    line. Then, $f$ can be uniquely determined up to global sign from
    $\set{ \abs{f(\tfrac{n}{4B})} }{ n \in \bbZ}$.  
\end{theorem}

\noindent
Another property of bandlimited functions that we will employ is that their ambiguity function is
compactly supported in frequency domain.

\begin{lemma}
\label{lem:pw_amb}
    Let $B > 0$ and $f \in \mathrm{PW}^2_B$. Then, $\calA f$ is uniformly continuous and $\supp
    \calA f \subset \bbR \times (-2B,2B)$.
\end{lemma}

\begin{proof}
    See appendix \ref{app:pw_amb}.
\end{proof}

\noindent
Therefore, we can consider $\phi \in L^2(\bbR)$ such that $\calA \phi$ does not vanish on 
$\bbR \times (-2B,2B)$ and reconstruct all $f \in \mathrm{PW}^2_B$ up to global phase from the STFT
magnitude measurements $\abs{\calV_\phi f }$. In what follows, we will show that, in fact, the STFT
phase retrieval problem is uniquely solvable for signals in $\mathrm{PW}^2_B$ under weaker
assumptions on $\calA \phi$. We remark that our uniqueness results are mainly of theoretical
interest and do not suggest a method for stable phase recovery.

\paragraph{Outline} This paper is divided into two main parts. First, in Section \ref{sec:real}, we
consider signals in the Paley--Wiener space which are real-valued on the real line and develop two
uniqueness results for this case: In particular, we show that if the ambiguity function of the
window is non-zero almost everywhere on a certain line segment in the time-frequency plane, then
all bandlimited signals are uniquely determined by their STFT magnitudes (Subsection
\ref{ssec:full}). In addition, we show that if the Fourier transform of the window is non-zero
almost everywhere on an open interval around the origin and if the window is real-valued itself,
then all bandlimited signals are uniquely determined by \emph{samples} of their STFT magnitudes
(Subsection \ref{ssec:sampled}). Secondly, in Section \ref{sec:complex}, we consider general
signals in the Paley--Wiener space and develop two uniqueness results in this setting. More
precisely, we show that if the ambiguity function of the window is non-zero almost everywhere on
two parallel line segments in the time-frequency plane that are sufficiently close together, then
uniqueness of phase retrieval from STFT magnitudes holds for all signals in $\mathrm{PW}^2_B$
(Subsection \ref{ssec:two}). In addition, we show that if the ambiguity function of the window does
not vanish on a single line segment in the time-frequency plane, then all bandlimited signals are
uniquely determined by their STFT magnitudes (Subsection \ref{ssec:single}). In Section
\ref{sec:conclusion}, we discuss our results for some examples of window classes.

\section{Real-valued signals}
\label{sec:real}

\subsection{Reconstruction from full measurements}
\label{ssec:full}

If $f \in \mathrm{PW}^2_B$, for some $B>0$, and the window $\phi \in L^2(\bbR)$ is such that $\calA
\phi (0,\omega) \neq 0$, for $\omega \in (-2B,2B)$, then one can use the Ambiguity Function
Relation to obtain $\calA f (0,\cdot)$ everywhere. Therefore, one can recover $\abs{f}$ on the real
line via Fourier inversion. If $f$ is real-valued on the real line, this is enough to recover $f$
everywhere up to global phase \cite{Thakur10} (see Theorem \ref{thm:tst}). The last insight is
particularly important such that we state it as a lemma.

\begin{lemma}
\label{lem:real_insight}
    Let $B > 0$ and $f \in \mathrm{PW}^2_B$ be real-valued on the real line. Then, $f$ is uniquely
    determined by $\set{ \abs{f(t)} }{t \in \bbR}$ up to global sign.
\end{lemma}

\begin{proof}
    This follows immediately from Theorem \ref{thm:tst}. Alternatively, one might
    consider the following argument: Let us assume without loss of generality that $f$ is
    non-trivial. By the Paley--Wiener theorem (see Theorem \ref{thm:pw}), $f$ is an entire
    function. The roots of a non-zero entire function are isolated and therefore there must exist
    an interval $I \subset \bbR \subset \bbC$ such that for all $t \in I$, $f(t) \neq 0$.
    Therefore, $\abs{f}$ agrees with $f$ up to global sign on $I$. In other words, $\abs{f}$ is the
    restriction of $f$ or $-f$ to the interval $I$ and thus analytically extending $\abs{f}$ from
    $I$ to $\bbC$ yields $f$ or $-f$.
\end{proof}

\noindent
Note that the same does not hold for general signals. Indeed, consider the following counterexample.

\begin{example}
Let $B > 0$, $f(z) = \sinc(Bz)$ and $g(z) = \sinc(Bz) \rme^{\pi\rmi B z}$, for $z \in \bbC$.
One can readily show that 
\begin{align*}
    f(z) &= \frac{1}{B} \int_{-B}^B \chi_{[-B/2,B/2]}(\xi) \rme^{2\pi\rmi\xi z} \dd \xi, \quad z \in \bbC, \\
    g(z) &= \frac{1}{B} \int_{-B}^B \chi_{[0,B]}(\xi) \rme^{2\pi\rmi\xi z} \dd \xi, \quad z \in \bbC.
\end{align*}
Therefore, we have $f,g \in \mathrm{PW}^2_B$. In addition, 
\[
    \abs{f(t)} = \abs{\sinc(Bt)} = \abs{g(t)}, \qquad t \in \bbR,
\]
but $f$ and $g$ do not agree up to global phase.
\end{example}

\noindent
Many more counterexamples may be constructed using Hadamard's factorisation theorem and ideas
similar to the ones in \cite{McDonald04}. We may now combine the lemma above with the ambiguity
function relation to derive the following theorem.

\begin{theorem}
\label{thm:real}
Let $B > 0$ and $\phi \in L^2(\bbR)$ such that 
\[
    \calA \phi(0,\omega) \neq 0, \qquad \mbox{for a.e.~} \omega \in (-2B,2B).
\]
Then, the following are equivalent for $f,g \in \mathrm{PW}^2_B$ real-valued on the real line:
\begin{enumerate}
    \item $f = \pm g$.
    \item $\abs{\calV_\phi f } = \abs{\calV_\phi g }$.
\end{enumerate}
\end{theorem}

\begin{proof}
    First, note that if $f = \pm g$, then it follows immediately that $\abs{\calV_\phi f } =
    \abs{\calV_\phi g }$. Secondly, suppose that $\abs{\calV_\phi f } = \abs{\calV_\phi g }$. It
    follows from the ambiguity function relation that 
    \[
        \calA f(x,\omega) \cdot \overline{\calA \phi (x,\omega)}
        = \calA g(x,\omega) \cdot \overline{\calA \phi (x,\omega)},
        \qquad x,\omega \in \bbR.
    \]
    Hence, by the assumption on the ambiguity function of the window, $\calA f(0,\omega) =
    \calA g (0,\omega)$, for a.e.~$\omega \in (-2B,2B)$. By the Paley--Wiener theorem,
    $f$ and $g$ are square integrable and thus $\calA f$ and $\calA g$ are (uniformly)
    continuous. Therefore, $\calA f(0,\omega) = \calA g(0,\omega)$, for all $\omega
    \in (-2B,2B)$. We know from Lemma \ref{lem:pw_amb} that $\supp \calA f,\supp \calA g \subset
    \bbR \times (-2B,2B)$ and consequently, that $\calA f(0,\cdot) = \calA g(0,\cdot)$. Since
    \[
        \calA f(0,\cdot) = \calF\left( \abs{f}^2 \right), \qquad \calA g(0,\cdot) = \calF\left( \abs{g}^2 \right),
    \]
    we have $\abs{f} = \abs{g}$. Applying Lemma \ref{lem:real_insight} yields the assertion.
\end{proof}

\begin{remark}
    Using basic Fourier-analytic results, one can show a statement which is similar to Theorem
    \ref{thm:real} for compactly supported, even, real-valued functions:

    \vspace{5pt}
    \noindent
    Let $B > 0$ and $\phi \in L^2(\bbR)$ such that 
    \[
        \calA \phi(x,0) \neq 0, \qquad \mbox{for a.e.~} x \in (-2B,2B).
    \]
    Then, the following are equivalent for $f,g \in L^2([-B,B])$ even and real-valued:
    \begin{enumerate}
        \item $f = \pm g$.
        \item $\abs{\calV_\phi f } = \abs{\calV_\phi g }$.
    \end{enumerate}
\end{remark}

\subsection{Reconstruction from sampled STFT measurements}
\label{ssec:sampled}

So far, we have not used the sampling theorem in \cite{Thakur10} for the reconstruction from
\emph{samples} of the STFT magnitudes. Let us consider the same setup as before and ask whether the
phase is uniquely determined \emph{by sampled data}. To the best of our knowledge, the following is
the first uniqueness result for phase retrieval from sampled STFT magnitude measurements.

\begin{theorem}
\label{thm:sampled}
    Let $B > 0$ and let $\phi \in L^2(\bbR)$ be a real-valued window such that 
    \[
        \left(\calF \phi\right)(\xi) \neq 0, \qquad \mbox{for a.e.~}\xi \in (-B,B).
    \]
    Then, the following are equivalent for $f,g \in \mathrm{PW}^2_B$ real-valued on the real line:
    \begin{enumerate}
        \item $f = \pm g$.
        \item $\abs{\calV_\phi f\left(\frac{n}{4B},0\right)} = \abs{\calV_\phi g\left(\frac{n}{4B},0\right)}$,
            for all $n \in \bbZ$.
    \end{enumerate}
\end{theorem}

\begin{proof}
First, note that if $f = \pm g$, then it follows immediately that
\[
    \abs{\calV_\phi f\left(\frac{n}{4B},0\right)} = \abs{\calV_\phi g\left(\frac{n}{4B},0\right)},
    \qquad n \in \bbZ.
\]
Secondly, assume that the above equation holds. Let us define $\phi^\#(t) :=
\overline{\phi(-t)}$, for $t \in \bbR$ (this simplifies to $\phi^\#(t) :=
\phi(-t)$ because $\phi$ is real-valued). Note that 
\[
    \calV_\phi f(x,0) = \int_\bbR f(t) \overline{\phi(t-x)} \dd t
    = \left( f \ast \phi^\# \right)(x), \qquad x \in \bbR,
\]
and hence
\[
    \abs{\left( f \ast \phi^\# \right)\left(\frac{n}{4B}\right)} =
    \abs{\left( g \ast \phi^\# \right)\left(\frac{n}{4B}\right)},
    \qquad n \in \bbZ.
\]
As $f,g \in \mathrm{PW}^2_B$, it follows from the convolution theorem that
$f \ast \phi^\#$ and $g \ast \phi^\#$ extend to functions in $\mathrm{PW}^1_B$. Indeed, as
$f \in \mathrm{PW}^2_B$, there exists $F' \in L^2([-B,B])$ such that 
\[
    f(z) = \int_{-B}^B F'(\xi) \rme^{2\pi\rmi\xi z} \dd \xi, \qquad z \in \bbC.
\]
Now, consider $F = F' \overline{\calF \phi} \in L^1([-B,B])$. Then, by the convolution theorem,
\[
    (f \ast \phi^\#)(t) = \int_{-B}^B F(\xi) \rme^{2 \pi \rmi \xi t} \dd \xi,
    \qquad t \in \bbR.
\]
Therefore, the analytic extensions of $f \ast \phi^\#$ and $g \ast \phi^\#$ belong to
$\mathrm{PW}_B^1$. In addition, $f \ast \phi^\#$ and $g \ast \phi^\#$ are real-valued such that it
follows from Theorem \ref{thm:tst} that $f \ast \phi^\# = \pm (g \ast \phi^\#)$. Consequently, we
have
\[
    \calF \left(f \ast \phi^\#\right) = \pm \calF \left(g \ast \phi^\#\right).
\]
By the assumption on the Fourier transform of the window, it follows
that $\calF f = \pm \calF g$ and hence, $f = \pm g$.
\end{proof}

\begin{remark} We observe the following:
    \begin{enumerate}
        \item The sampling rate only depends on the bandwidth of the signals and is exactly twice
            the Nyquist rate.
        \item For the Gaussian $\phi(t) := \rme^{-\pi t^2}$, $t \in \bbR$, we can readily see that 
            $\calF \phi$ is non-zero everywhere. In addition, the Gaussian is real-valued such that
            the theorem above implies that all bandlimited signals are uniquely determined by
            samples of their STFT magnitudes with Gaussian window (also called Gabor transform
            magnitudes).
        \item We use the uniform sampling sequence $X = \{ \tfrac{n}{4B} \}_{n \in \bbZ} \subset \bbR$, for 
        convenience of notation. In fact, our result still holds if one replaces $X =
        \{ \tfrac{n}{4B} \}_{n \in \bbZ}$ by any separated, uniformly dense sampling sequence with
        density lower bounded by $4B$ \cite{Thakur10}.
        \item While the STFT is complex-valued, we employ only real-valued information on one line of
        the time-frequency plane. On this line, \emph{sign retrieval} suffices for the uniqueness
        result to hold.
    \end{enumerate}
\end{remark}

%

\section{Complex-valued signals}
\label{sec:complex}

\subsection{Using the ambiguity function on two line segments}
\label{ssec:two}

We have seen that for complex-valued $f \in \mathrm{PW}^2_B$, it is not true that $f$ is
uniquely determined up to global phase by $\set{\abs{f(t)}}{t \in \bbR}$. Therefore, we need to
change our strategy to deduce a general uniqueness result. One can, for instance, impose slightly
stronger assumptions on the ambiguity function of the window to show that in this case one may
recover $f$ uniquely up to global phase from $\abs{\calV_\phi f}$.

\begin{theorem}
\label{thm:complex}
Let $B > 0$, $c \in (0, \tfrac{1}{2B}]$ and $\phi \in L^2(\bbR)$ such that 
\[
    \calA \phi(0,\omega) \neq 0 \qquad \mbox{and} \qquad \calA \phi(c,\omega) \neq 0,
\]
for a.e.~$\omega \in (-2B,2B)$. Then, the following are equivalent for $f,g \in \mathrm{PW}^2_B$:
\begin{enumerate}
    \item $f = \rme^{\rmi\alpha} g$, for some $\alpha \in \bbR$.
    \item $\abs{\calV_\phi f} = \abs{\calV_\phi g}$.
\end{enumerate}
\end{theorem}

\begin{proof}
    First, note that if $f = \rme^{\rmi\alpha} g$, for some $\alpha \in \bbR$, then it follows
    immediately that $\abs{\calV_\phi f} = \abs{\calV_\phi g}$. Secondly, suppose that
    $\abs{\calV_\phi f} = \abs{\calV_\phi g}$. Let us also assume without loss of generality 
    that $f$ and $g$ are non-zero. As in the proof of Theorem \ref{thm:real}, we find that 
    $\calA f(0,\cdot) = \calA g(0,\cdot)$ and $\calA f(c,\cdot) = \calA g(c,\cdot)$. Therefore,
    $\abs{f} = \abs{g}$ on $\bbR$ and 
    \[
        f(t) \overline{f(t-c)} = g(t) \overline{g(t-c)}, \qquad t \in \bbR.
    \]
    By the Paley--Wiener theorem, $f$ and $g$ are entire functions. Therefore, we know that $f$ and
    $g$ have a countable number of roots (as their roots are isolated). In particular, there exists
    some $t_0 \in \bbR$ such that for all $n \in \bbZ$, we have 
    \[
        f(t_0 + n c) \neq 0 \qquad \mbox{and} \qquad g(t_0 + nc) \neq 0.
    \]
    Now, let us set $\alpha \in (-\pi,\pi]$ to be such that 
    \[
        f(t_0) = \rme^{\rmi \alpha} g(t_0).
    \]
    Then, we can use the relation 
    \[
        f(t) \overline{f(t-c)} = g(t) \overline{g(t-c)}, \qquad t \in \bbR,
    \]
    to recursively find that 
    \[
        f(t_0 + n c) = \rme^{\rmi \alpha} g(t_0 + nc), \qquad n \in \bbZ.
    \]
    Finally, since $f,g \in \mathrm{PW}^2_B$ and $c \leq \frac{1}{2B}$, it follows that
    $f(t_0 + \cdot), g(t_0 + \cdot) \in \mathrm{PW}^2_{\frac{1}{2c}}$. Therefore,
    we deduce from the WSK sampling theorem (see Theorem \ref{thm:wsk}) that
    \[
        f(t_0 + t) = \sum_{n \in \bbZ} f(t_0 + n c) \sinc\left( \frac{t}{c} - n \right)
        = \sum_{n \in \bbZ} \rme^{\rmi \alpha} g(t_0 + nc) \sinc\left( \frac{t}{c} - n \right)
        = \rme^{\rmi \alpha } g(t_0 + t), 
    \]
    for all $t \in \bbR$. Hence, we conclude that $f = \rme^{\rmi \alpha} g$.
\end{proof}

\begin{remark}
    We can apply basic Fourier analysis to develop a result similar to Theorem \ref{thm:complex}
    for compactly supported functions:

    \vspace{5pt}
    \noindent
    Let $B > 0$, $c \in (0, \tfrac{1}{2B}]$ and $\phi \in L^2(\bbR)$ such that 
    \[
        \calA \phi(x,0) \neq 0 \qquad \mbox{and} \qquad \calA \phi(x,c) \neq 0,
    \]
    for a.e.~$x \in (-2B,2B)$. Then, the following are equivalent for $f,g \in L^2([-B,B])$:
    \begin{enumerate}
        \item $f = \rme^{\rmi\alpha} g$, for some $\alpha \in \bbR$.
        \item $\abs{\calV_\phi f} = \abs{\calV_\phi g}$.
    \end{enumerate}
\end{remark}

\noindent
We note that the requirement $c \leq \frac{1}{2B}$ is linked to the use of the WSK sampling theorem
in the proof of the above result. The following example shows that it is necessary in order for $f
\in \mathrm{PW}^2_B$ to be uniquely determined up to global phase by $\abs{f(t)}$ and $f(t)
\overline{f(t-c)}$, for $t \in \bbR$.

\begin{example}
Let $\epsilon,B > 0$ and $c = \frac{1}{2B-\epsilon} > \frac{1}{2B}$. Consider $f(z) =
\sinc(\epsilon z) \rme^{\pi \rmi (2B-\epsilon) z}$ as well as $g(z) = \sinc(\epsilon z)
\rme^{- \pi \rmi (2B-\epsilon) z}$, for $z \in \bbC$. Then, it is readily seen that 
\[
    f(z) = \frac{1}{\epsilon} \int_{-B}^B \chi_{[B-\epsilon,B]}(\xi) \rme^{2\pi\rmi\xi z} \dd \xi,
    \qquad z \in \bbC,
\]
as well as
\[
    g(z) = \frac{1}{\epsilon} \int_{-B}^B \chi_{[-B,-B+\epsilon]}(\xi) \rme^{2\pi\rmi\xi z} \dd \xi,
    \qquad z \in \bbC.
\]
Therefore, we have $f,g \in \mathrm{PW}_B^2$. In addition, 
\[
    \abs{f(t)} = \abs{\sinc(\epsilon t)} = \abs{g(t)},
\]
as well as 
\begin{align*}
    f(t) \overline{f(t - c)} &= \sinc(\epsilon t) \sinc(\epsilon(t-c)) \rme^{\pi \rmi (2B-\epsilon) c}
    = - \sinc(\epsilon t) \sinc(\epsilon(t-c)), \\
    &= \sinc(\epsilon t) \sinc(\epsilon(t-c)) \rme^{-\pi \rmi (2B-\epsilon) c}
    = g(t) \overline{g(t - c)},
\end{align*}
hold, for $t \in \bbR$. However, $f$ and $g$ do not agree up to global phase.
\end{example}

\noindent
Many more examples may be constructed using Hadamard's factorisation theorem and ideas similar to 
the ones in \cite{McDonald04}.

\subsection{Using the ambiguity function on a single line segment}
\label{ssec:single}

We can approach the reconstruction of general bandlimited functions from their STFT magnitude
measurements from a slightly different angle and obtain another uniqueness result. The following
statement is neither stronger nor weaker than Theorem \ref{thm:complex}: Indeed, it has the
advantage that one only needs to assume that the ambiguity function of the window does not vanish
on a single line segment while having the disadvantage that one has to make this assumption
pointwise and not in an $L^2$-sense.

\begin{theorem}
\label{thm:complex_new}
Let $B > 0$ and $\phi \in L^2(\bbR)$ be such that 
\[
    \calA \phi(0,\omega) \neq 0, \qquad \omega \in [-2B,2B].
\]
Then, the following are equivalent for $f,g \in \mathrm{PW}_B^2$:
\begin{enumerate}
    \item $f = \rme^{\rmi\alpha} g$, for some $\alpha \in \bbR$.
    \item $\abs{\calV_\phi f} = \abs{\calV_\phi g}$.
\end{enumerate}
\end{theorem}

\begin{proof}
    First, note that if $f = \rme^{\rmi\alpha} g$, for some $\alpha \in \bbR$, then it follows
    immediately that $\abs{\calV_\phi f} = \abs{\calV_\phi g}$. Secondly, suppose that
    $\abs{\calV_\phi f} = \abs{\calV_\phi g}$ and assume without loss of generality that $f$ and 
    $g$ are non-zero. Note that $\calA \phi$ is continuous such that $\abs{\calA \phi }$ is
    continuous and by assumption $\abs{\calA \phi(0,\omega)} > 0$, for $\omega \in [-2B,2B]$.
    By the extreme value theorem, there exists a positive constant $\Delta > 0$ such that
    $\abs{\calA \phi(0,\omega)} \geq \Delta$, for $\omega \in [-2B,2B]$. As
    $\abs{\calA \phi}$ is uniformly continuous, there exists a $\delta > 0$ such that 
    \[
        \abs{\calA \phi(x,\omega)} >
        \tfrac{\Delta}{2}, \qquad (x,\omega) \in (-\delta,\delta) \times [-2B,2B].
    \]
    In particular, it follows that 
    \[
        \calA \phi(x,\omega) \neq 0, \qquad (x,\omega) \in (-\delta,\delta)
        \times [-2B,2B].
    \]
    Lemma \ref{lem:amb} further implies
    \[
        \calA f(x,\omega) = \calA g (x,\omega), \qquad (x,\omega) \in (-\delta,\delta) \times
        [-2B,2B].
    \]
    Hence, by employing Lemma \ref{lem:pw_amb}, we deduce that $ \calA f(x,\omega) = \calA g (x,\omega)$,
    for $(x,\omega) \in (-\delta,\delta) \times \bbR$. By Fourier inversion, we have
    \[
        f(t)\overline{f(t-c)} = g(t)\overline{g(t-c)}, \qquad t \in \bbR,
    \]
    for $c \in (-\delta,\delta)$. As $f$ and $g$ are entire functions and assumed to be non-zero, 
    there exists a $t_0 \in \bbR$ such that $f(t_0), g(t_0) \neq 0$. As $\abs{f(t_0)} =
    \abs{g(t_0)}$, it follows that there exists an $\alpha \in \bbR$ such that $f(t_0) =
    \rme^{\rmi \alpha} g(t_0)$. This implies that $f(t) = \rme^{\rmi \alpha} g(t)$ for all 
    $t \in (t_0-\delta, t_0+\delta)$. As $f$ and $g$ are entire, we conclude that $f =
    \rme^{\rmi \alpha} g$.
\end{proof}

\begin{remark}
    As before, we can make a similar statement as the above for compactly supported functions:

    \vspace{5pt}
    \noindent
    Let $B > 0$ and $\phi \in L^2(\bbR)$ be such that 
    \[
        \calA \phi(x,0) \neq 0, \qquad x \in [-2B,2B].
    \]
    Then, the following are equivalent for $f,g \in L^2([-B,B])$:
    \begin{enumerate}
        \item $f = \rme^{\rmi\alpha} g$, for some $\alpha \in \bbR$.
        \item $\abs{\calV_\phi f} = \abs{\calV_\phi g}$.
    \end{enumerate}
\end{remark}

\section{Examples}
\label{sec:conclusion}

In the following, we want to consider different windows and their ambiguity functions in order to 
put the results which we have developed in context. We start by considering the most
well-known window in time-frequency analysis: The \emph{Gaussian window} $\phi(t) :=
\rme^{-\pi t^2}$, for $t \in \bbR$ (see Figure \ref{fig:gaussian}). For the Gaussian, one can show that
\begin{figure}[b!]
    \centering
    \subfloat[][]{\includegraphics[width=.45\linewidth]{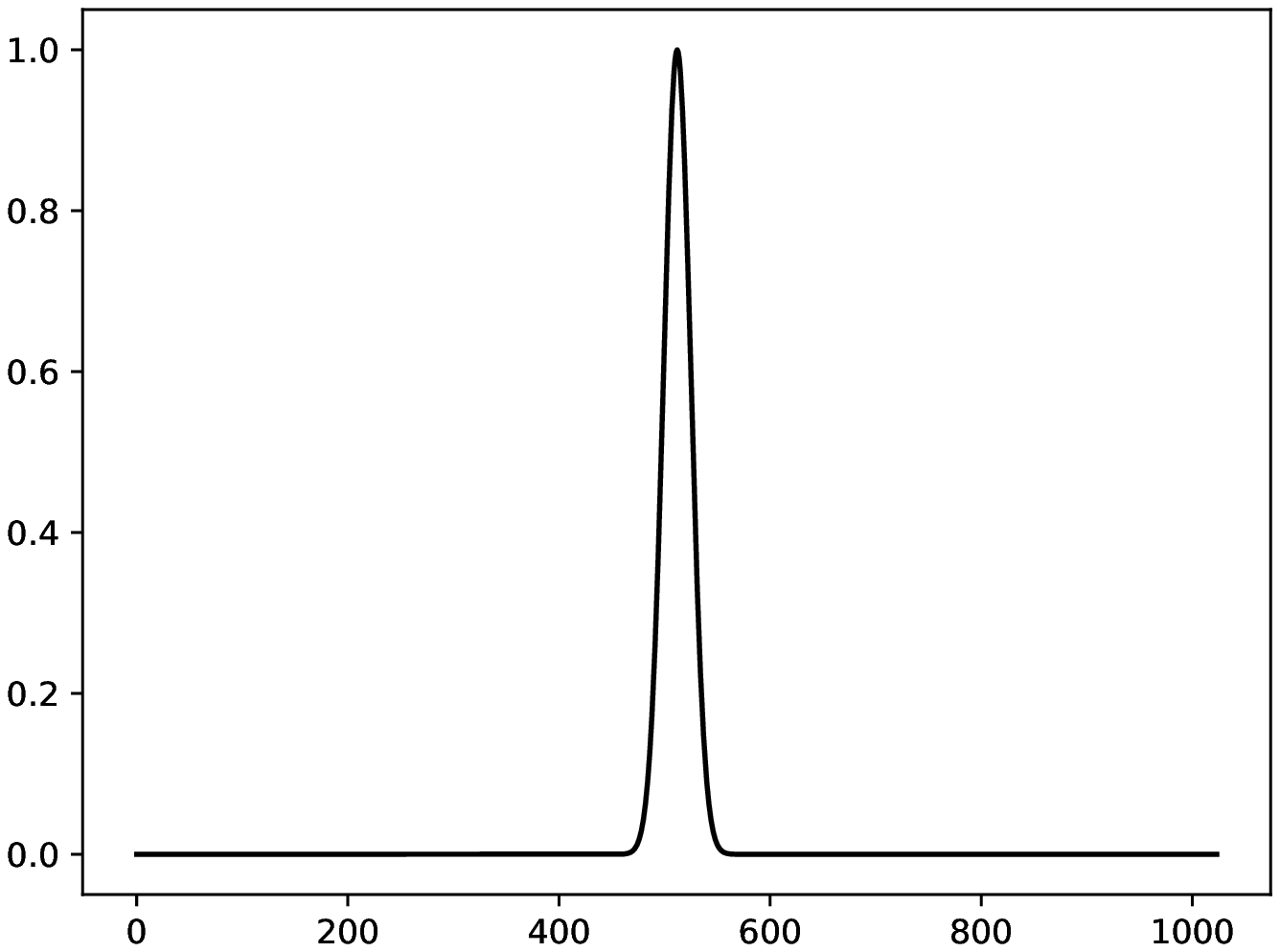}}
    \subfloat[][]{\includegraphics[width=.45\linewidth]{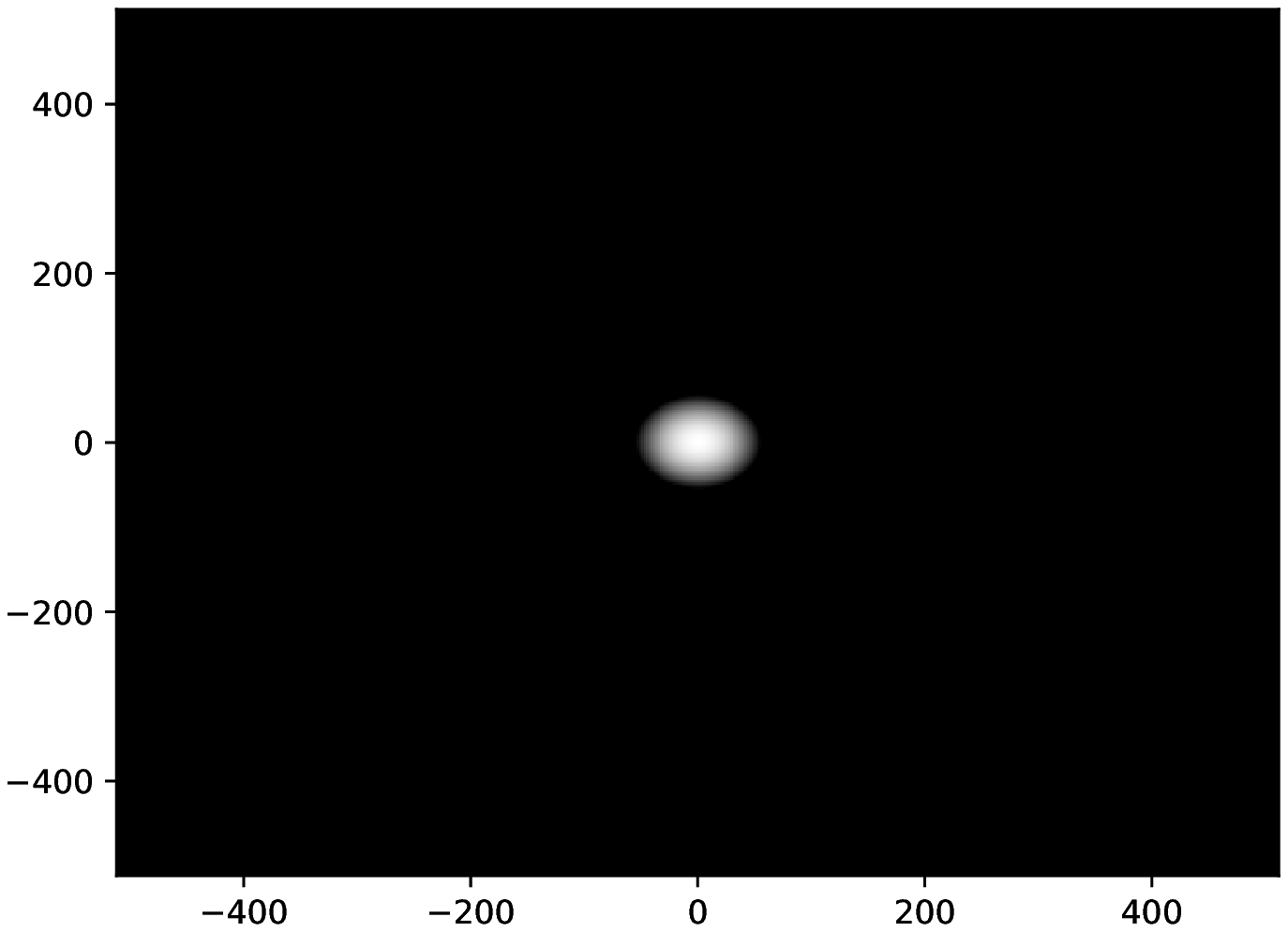}}
    \caption{Picture of a discretisation of the Gaussian window and its ambiguity function.}
    \label{fig:gaussian}
\end{figure}
\[
\calA \phi(x,\omega) = \frac{1}{\sqrt{2}}
    \rme^{-\frac{\pi}{2}\left( x^2 + \omega^2 \right)},
    \qquad x,\omega \in \bbR.
\]
Therefore, $\calA \phi$ is nowhere vanishing and all $f \in L^2(\bbR)$ may be uniquely recovered
up to global phase from their STFT magnitude measurements $\abs{\calV_\phi f}$. Note that
this already follows from the classical theory about uniqueness of STFT phase retrieval (Lemma
\ref{lem:classical}). One could be tempted to believe that the only windows for which 
the ambiguity functions are non-zero everywhere are the generalised Gaussians $\rme^q$, where $q$
is a polynomial of degree two. This belief is wrong, however, as was recently shown in
\cite{Groechenig19}, and one can in fact construct more functions $\phi \in L^2(\bbR)$ such that 
$\calA \phi \neq 0$ everywhere. Finally, note that $\phi$ is real-valued and that the Fourier
transform of $\phi$ is also a Gaussian. In particular, $\calF \phi$ vanishes nowhere and it
follows from Theorem \ref{thm:sampled} that for all $B > 0$, one can recover all $f \in
\mathrm{PW}^2_B$ that are real-valued on the real line up to global sign from the sampled
measurements $\abs{\calV_\phi f(\tfrac{n}{4B},0)}$, for $n \in \bbZ$.

\vspace{5pt}
\noindent 
The next class of window functions we want to study is that of the Hermite functions. We define 
the monomials 
\[
    e_n(z) := \sqrt{ \frac{\pi^n}{n!} } z^n, \qquad z \in \bbC,
\]
for $n \in \bbZ_{\geq 0}$. One can show that these monomials form an orthonormal basis of the Fock 
space $\calF^2(\bbC)$ \cite{Groechenig01}. The pre-images of these monomials under the Bargmann
transform $\calB : L^2(\bbR) \to \calF^2(\bbC)$ are called \emph{Hermite functions} and we write $H_n
:= \calB^{-1} e_n$, for $n \in \bbZ_{\geq 0}$. The ambiguity function of the Hermite functions can be
expressed in terms of the \emph{Laguerre polynomials} 
\[
    L_k^{(j)}(t) := \sum_{m = 0}^k \frac{(k+j)!}{(k-m)!(j+m)!} \frac{(-t)^m}{m!},
    \qquad t \in \bbR,
\]
where $k,j \in \bbZ_{\geq 0}$ \cite{Folland89}. In particular, we have for all $x,\omega \in \bbR$ and
$z = x + \rmi \omega$ that (see Figure \ref{fig:hermite1} and Figure \ref{fig:hermite7} for an
illustration):
\begin{figure}
    \centering
    \subfloat[][]{\includegraphics[width=.45\linewidth]{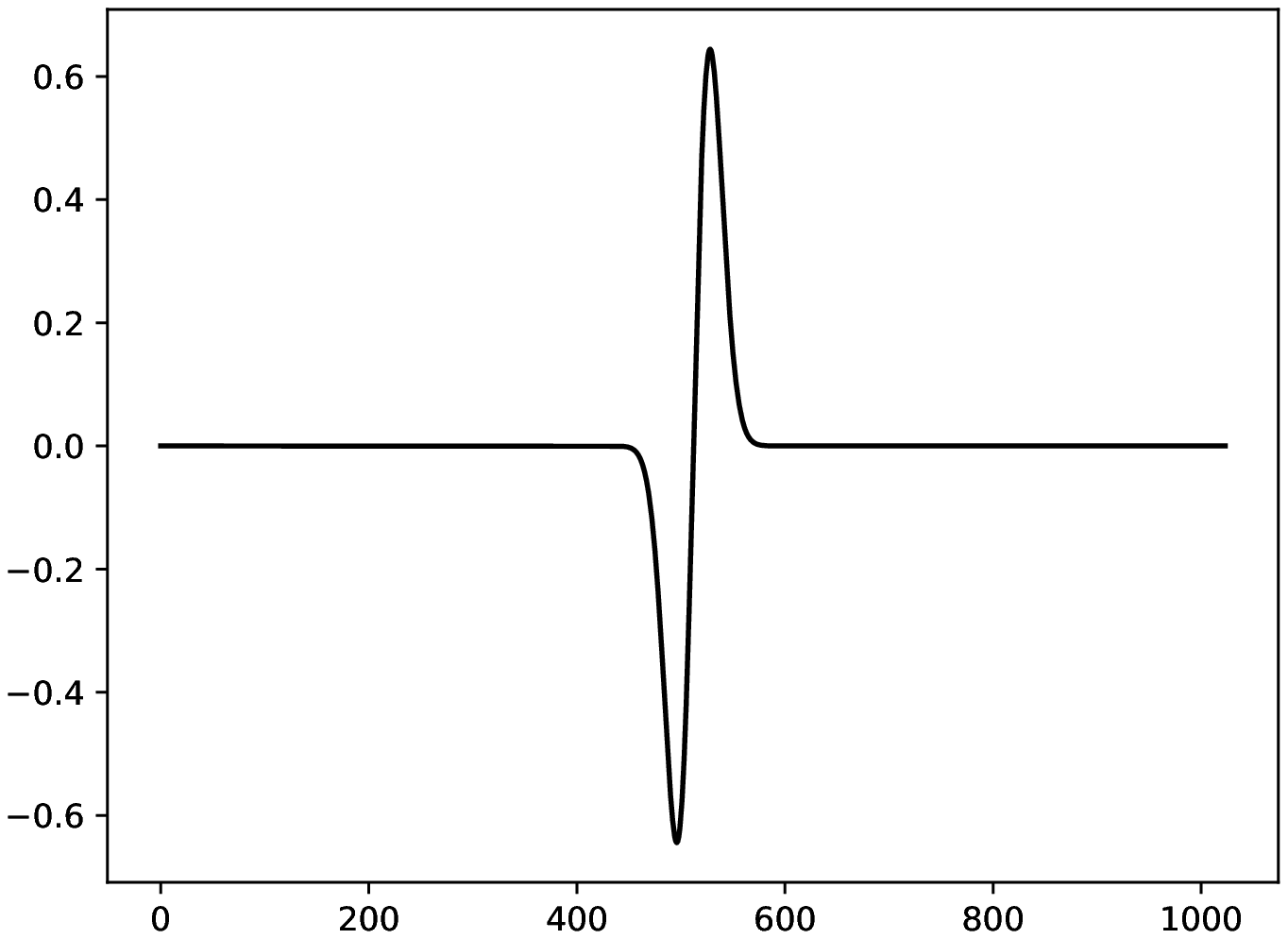}}
    \subfloat[][]{\includegraphics[width=.45\linewidth]{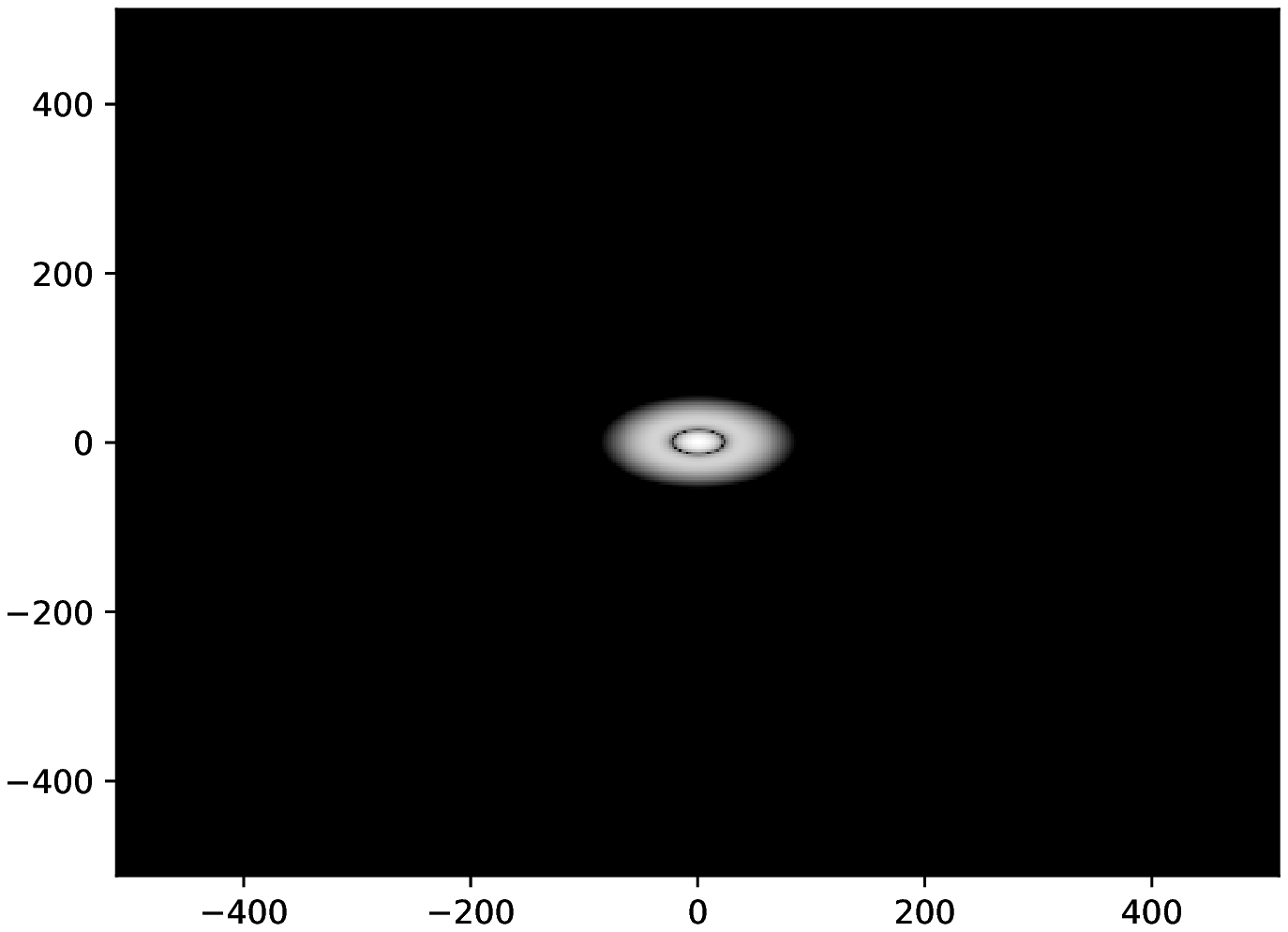}}
    \caption{Picture of a discretisation of the Hermite function $H_1$ and its ambiguity function.}
    \label{fig:hermite1}
\end{figure}
\begin{figure}
    \centering
    \subfloat[][]{\includegraphics[width=.45\linewidth]{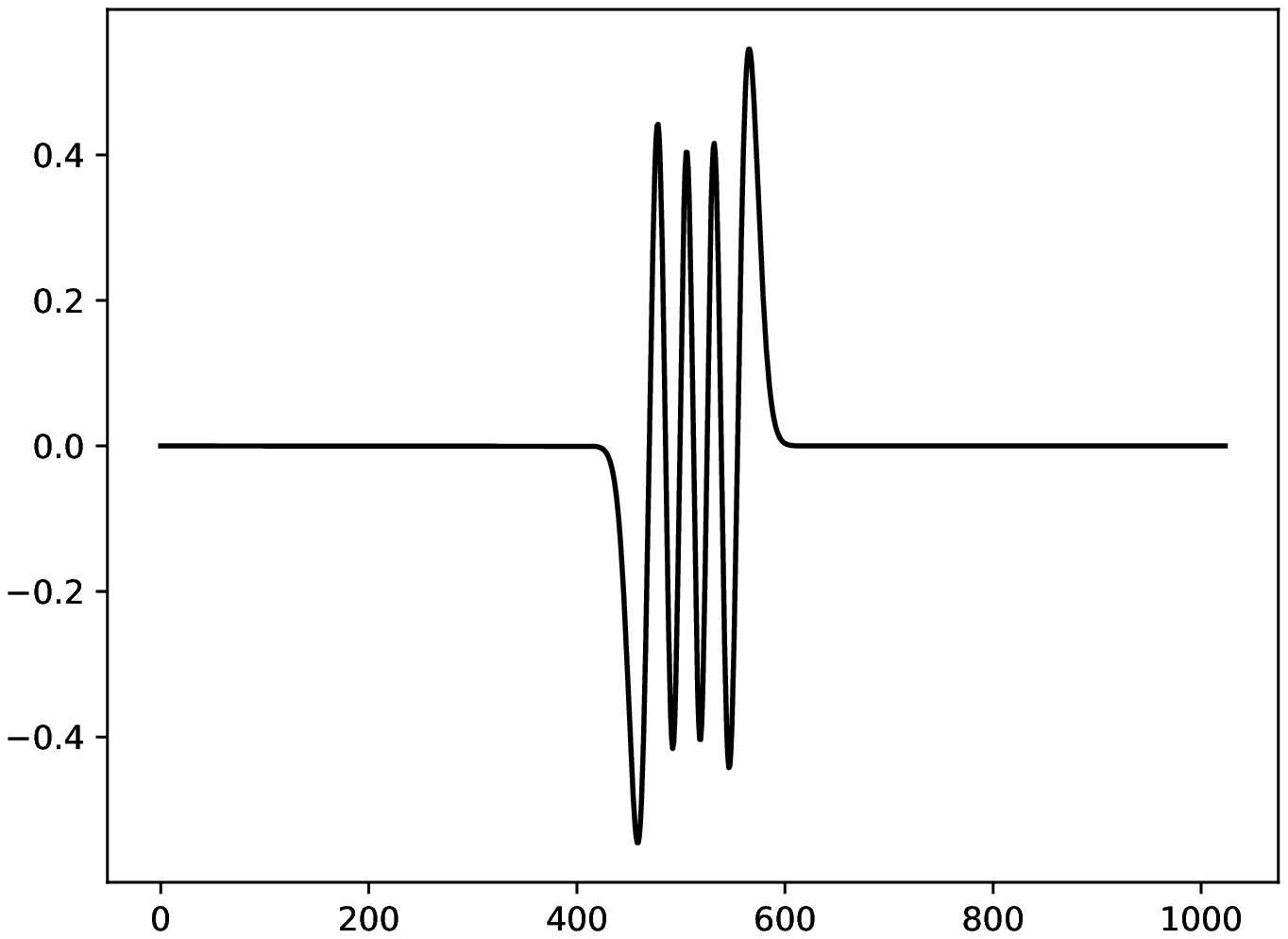}}
    \subfloat[][]{\includegraphics[width=.45\linewidth]{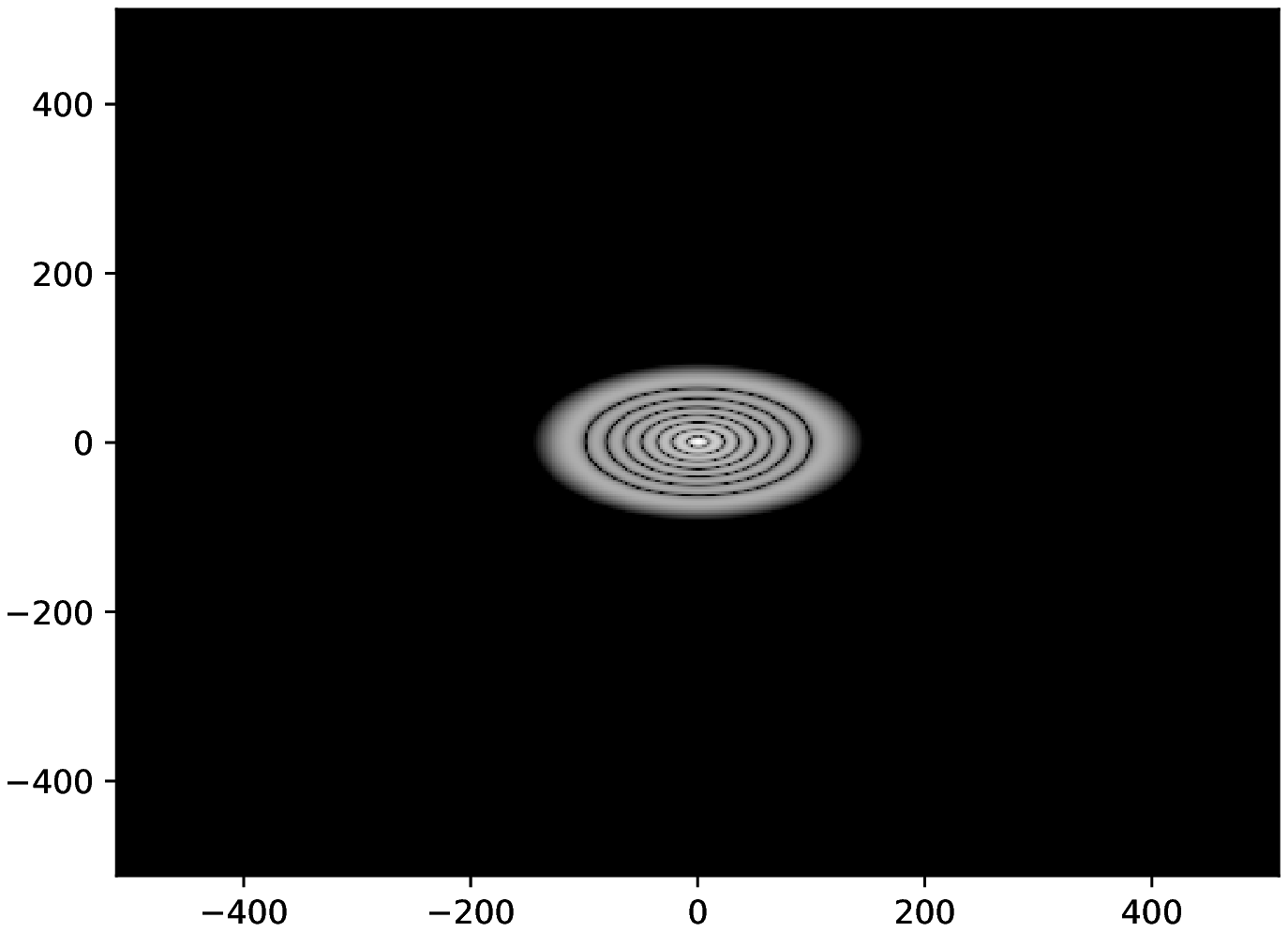}}
    \caption{Picture of a discretisation of the Hermite function $H_7$ and its ambiguity function.}
    \label{fig:hermite7}
\end{figure}
\[
    \calA H_n(x,\omega) = \rme^{-\frac{\pi}{2}\abs{z}^2} L_n^{(0)}(\pi \abs{z}^2)
    = \rme^{-\frac{\pi}{2}\abs{z}^2} \sum_{m = 0}^n \binom{n}{m} \frac{(-\pi)^m \abs{z}^{2m}}{m!}.
\]
Therefore, the set of roots of $\calA H_n$ consists of concentric rings around the origin of the
time-frequency plane. The radius of these rings is determined by the positive roots 
of the $n$-th Laguerre polynomial $L_n^{(0)}$. In particular, $\calA H_n$ is non-zero almost 
everywhere and it follows from Lemma \ref{lem:classical} that all $f \in L^2(\bbR)$ are uniquely 
determined up to global phase by their STFT measurements. As for the Gaussian case, uniqueness of
STFT phase retrieval already follows from the ambiguity function relation. Note that the
Fourier transform of the Hermite function $H_n$ is $(-\rmi)^n H_n$ \cite{Groechenig01}. In
addition, one can show that $h_n(t) := \rme^{\pi t^2} H_n(t)$, $t \in \bbR$, is a polynomial of
degree $n$ \cite{Folland89}. It follows
that $\calF H_n$ has only finitely many roots and thus $\calF H_n$ is non-zero almost everywhere. In 
addition, as $H_n$ is real-valued, it follows from Theorem \ref{thm:sampled} that for all $B > 0$,
it holds that all $f \in \mathrm{PW}_B^2$ that are real-valued on the real line can be recovered up to
global sign from the sampled measurements $\abs{\calV_{H_n} f(\tfrac{n}{4B},0)}$, for $n \in \bbZ$.

\vspace{5pt}
\noindent
The third class of window functions, we consider is that of compactly supported window 
functions. This class includes all windows commonly used in practice. Consider for instance 
the \emph{rectangular window}
\[
\phi(t) = \begin{cases}
    1 & \mbox{if } t \in [-1,1], \\
    0 & \mbox{else},
\end{cases}
\]
or the \emph{Hanning window} $\phi := \cos^2 \chi_{[-\pi/2,\pi/2]}$. If $\phi \in L^2(\bbR)$ is
compactly supported, then for any fixed $x \in \mathbb{R}$, the function $\omega \mapsto
\calA \phi(x,\omega)$ is bandlimited. In particular, $\omega \mapsto \calA \phi(x,\omega)$ extends
to an analytic function on the complex plane. Therefore, $z \mapsto \calA \phi(x,z)$ is either zero
or has merely isolated zeroes on the real line. Hence, we can conclude that for all $x \in \bbR$
such that $\omega \mapsto \calA \phi(x,\omega)$ is not the trivial map, it holds that
\begin{figure}
    \centering
    \subfloat[][]{\includegraphics[width=.45\linewidth]{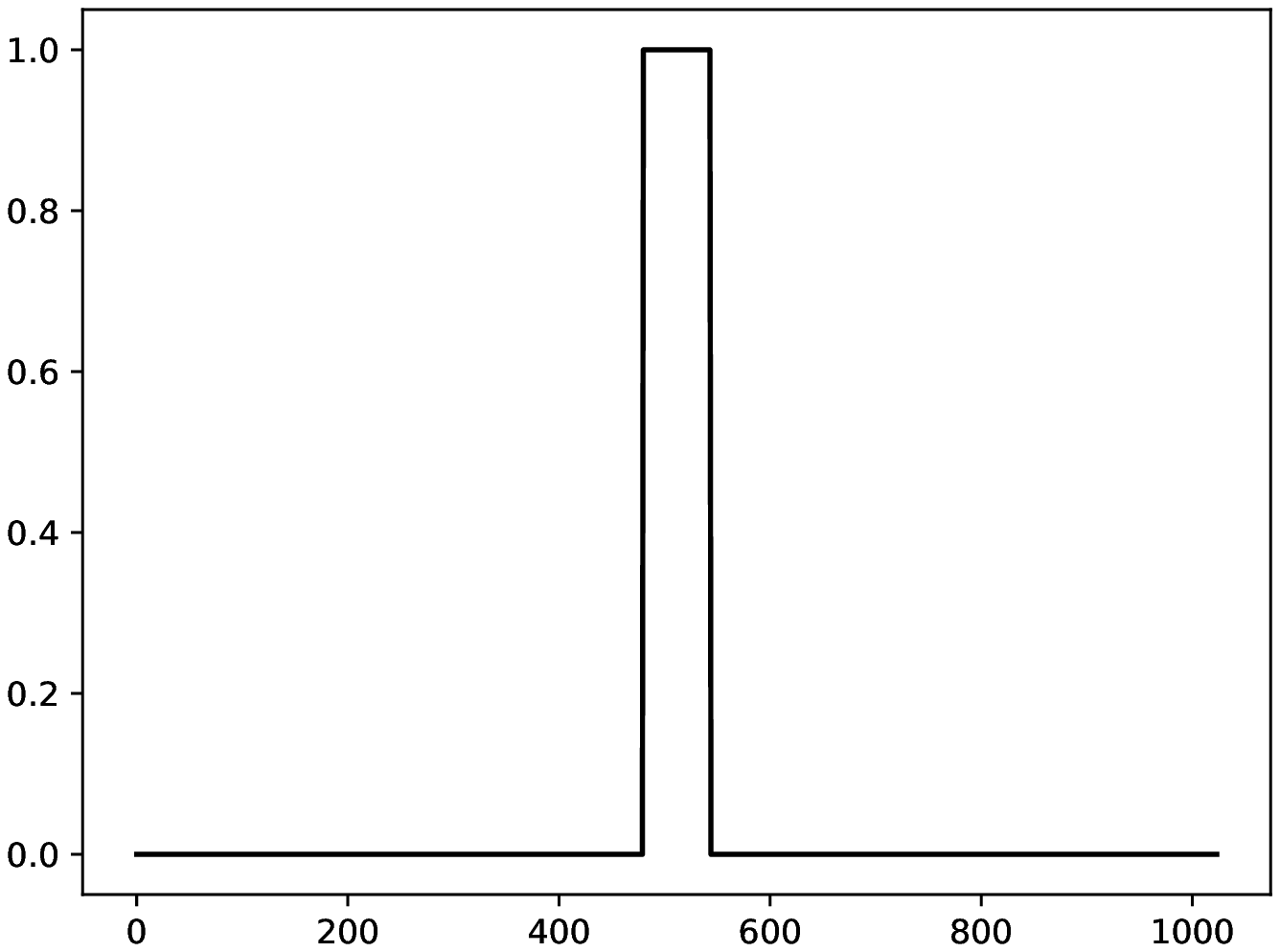}}
    \subfloat[][]{\includegraphics[width=.45\linewidth]{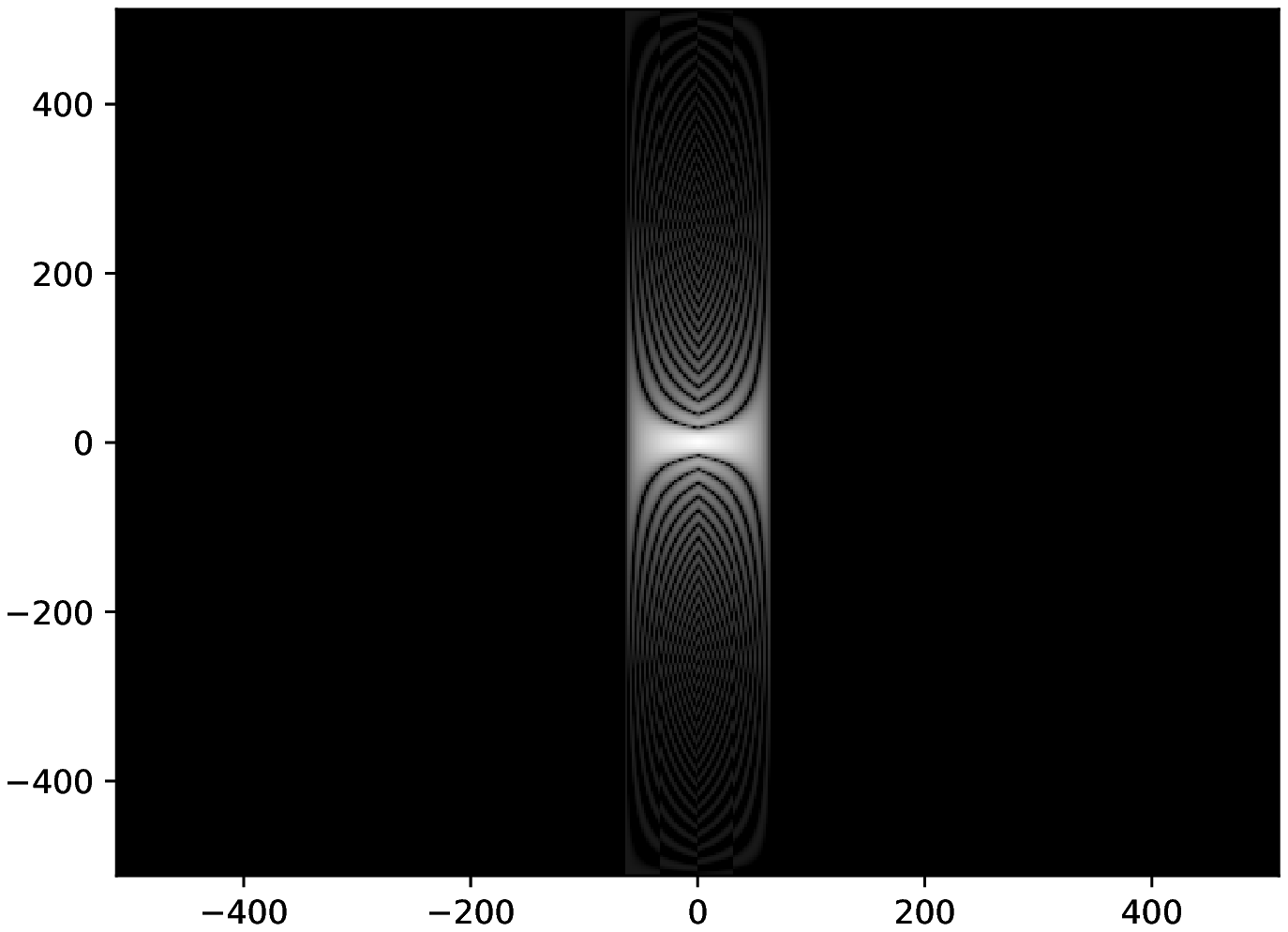}}
    \caption{Picture of a discretisation of the rectangular window and its ambiguity function.}
    \label{fig:rectangular}
\end{figure}
\begin{figure}
    \centering
    \subfloat[][]{\includegraphics[width=.45\linewidth]{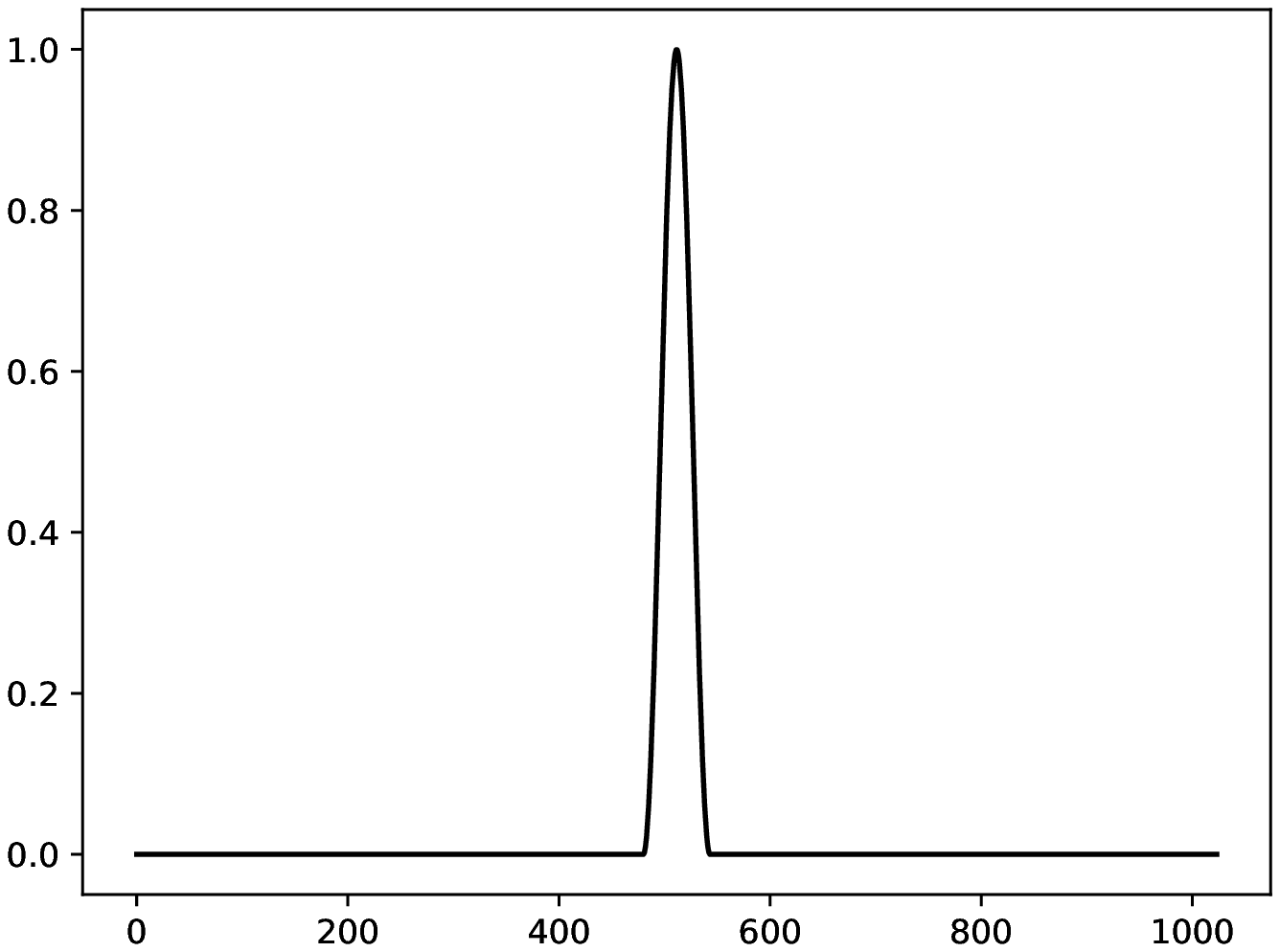}}
    \subfloat[][]{\includegraphics[width=.45\linewidth]{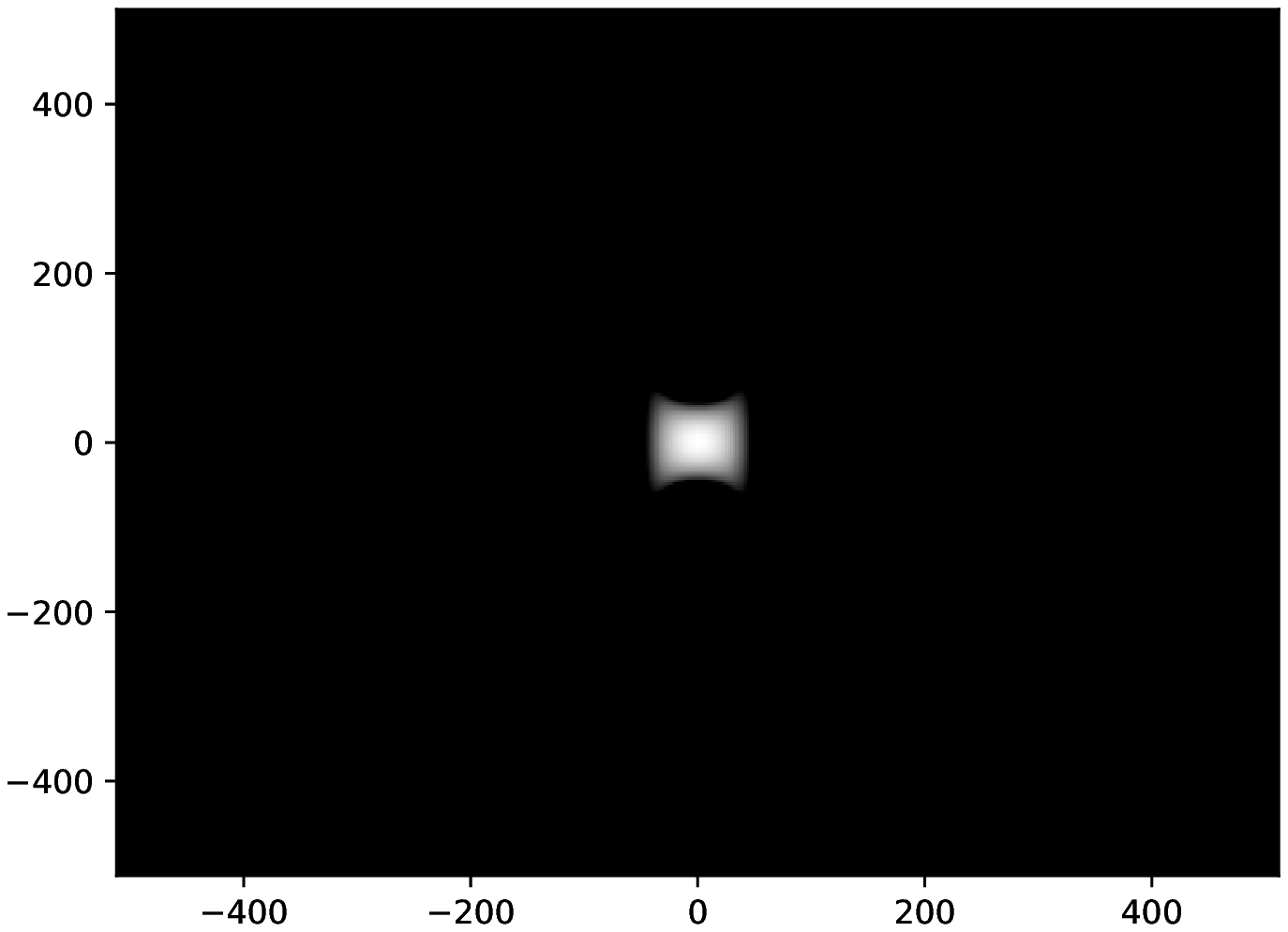}}
    \caption{Picture of a discretisation of the Hanning window and its ambiguity function.}
    \label{fig:hanning}
\end{figure}
\[
    \calA \phi(x,\omega) \neq 0, \qquad \mbox{for a.e.~}\omega \in \bbR.
\]
For a depiction of the ambiguity functions of the rectangular and Hanning windows see
Figure \ref{fig:rectangular} and Figure \ref{fig:hanning}, respectively. As $\calA \phi(0,\omega) =
\calF(\abs{\phi}^2)(\omega)$, it follows that $\calA \phi (0,\omega) \neq 0$ for a.e. $\omega \in
\bbR$, as long as $\phi$ is not the trivial window. Therefore, Theorem \ref{thm:real} implies that
for all $B > 0$, it holds that all $f \in \mathrm{PW}_B^2$ which are real-valued on the real line
are uniquely determined up to global sign by their STFT measurements. We can say more, however. If
$\phi$ is not the trivial window, then $\calA \phi (0,0) = \norm{\phi}_2 > 0$. As $\calA \phi $ is
continuous, it follows that for all $x > 0$ which are small enough, $\calA \phi (x,0) > 0$.
Therefore, $\calA \phi (x,\omega) \neq 0$ for almost every $\omega \in \bbR$. By Theorem
\ref{thm:complex}, we find that for all $B > 0$, it holds that all $f \in \mathrm{PW}_B^2$ are
uniquely determined up to global phase by their STFT magnitudes.

\paragraph{Acknowledgements} The authors would like to thank Giovanni S.~Alberti for fruitful
discussions and acknowledge funding through SNF Grant 200021\_184698.

\appendix

\section{The ambiguity function relation}
\label{app:ambiguity}

\begin{proof}[Proof of the ambiguity function relation]
Using the family $\{f_x \in L^1(\bbR) \,|\, x \in \bbR \}$ given by 
\[
    f_x(t) = f(t) \overline{\phi(t-x)}, \qquad t \in \bbR,
\]
for $x \in \bbR$, allows us to write 
\[
    \calV_\phi f(x,\omega) = \calF f_{x}(\omega), \qquad x,\omega \in \bbR.
\]
In addition, we can readily see that 
\[
    \overline{\calF f_x(\omega)} = \calF f_x^\#(\omega), \qquad x,\omega \in \bbR,
\]
where $f_x^\#(t) = \overline{f_x(-t)}$, for $t,x \in \bbR$. Let $x \in \bbR$
be fixed but arbitrary. As $f_x \in L^1(\bbR)$, it follows from the convolution theorem that 
\[
    \abs{\calV_\phi f(x,\omega)}^2 = \calV_\phi f(x,\omega) \overline{\calV_\phi f(x,\omega)}
    = \calF\left( f_x \ast f_x^\# \right)(\omega), \qquad \omega \in \bbR.
\]
As $\calV_\phi f \in L^2(\bbR^2)$ (by the orthogonality relations of the STFT
\cite{Groechenig01}), it follows that the STFT magnitude measurements squared are in $L^1$ and thus
the Fourier inversion theorem implies that 
\[
    \calF\left( \abs{\calV_\phi f(x,\cdot)}^2 \right)(-x')
    = \left( f_x \ast f_x^\# \right)(x')
    = \int_\bbR f(t) \overline{f(t-x') \phi(t-x)} \phi(t-x-x') \dd t,
\]
for $-x' \in \bbR$. Finally, we can see the above as a function in $x \in \bbR$ and note that 
\begin{align*}
    \int_\bbR \abs{\int_\bbR f(t) \overline{f(t-x') \phi(t-x)} \phi(t-x-x') \dd t} \dd x
    &\leq \int_\bbR \int_\bbR \abs{f(t) f(t-x') \phi(t-x) \phi(t-x-x')} \dd t \dd x \\
    &= \int_\bbR \abs{f(t) f(t-x')} \int_\bbR \abs{\phi(t-x) \phi(t-x-x')} \dd x \dd t \\
    &= \int_\bbR \abs{f(t) f(t-x')} \int_\bbR \abs{\phi(x) \phi(x-x')} \dd x \dd t \\
    &\leq \norm{f}_2^2 \norm{\phi}_2^2 < \infty,
\end{align*}
for $x' \in \bbR$, by the triangle inequality, Tonelli's theorem, a change of variables and
Cauchy--Schwarz. Therefore, we may take the Fourier transform in $x$ and obtain 
\begin{align*}
    \calF\left( \abs{\calV_\phi f}^2 \right)(\omega',-x')
    &= \int_\bbR \int_\bbR f(t) \overline{f(t-x') \phi(t-x)} \phi(t-x-x') \rme^{-2\pi\rmi x\omega'} \dd t \dd x \\
    &= \int_\bbR f(t) \overline{f(t-x')} \rme^{-2\pi\rmi t\omega'}
        \int_\bbR \overline{\phi(t-x)} \phi(t-x-x') \rme^{2\pi\rmi (t-x)\omega'} \dd x \dd t \\
    &= \calV_f f(x', \omega') \overline{\calV_\phi \phi(x',\omega')},
\end{align*}
for $x', \omega' \in \bbR$, by Fubini's theorem and a change of variables. Finally, note that 
all the equalities in this proof are actual equalities as all functions that we compare are 
continuous functions by virtue of them being Fourier transforms of $L^1$ functions.
\end{proof}

\section{The ambiguity function of a Paley--Wiener function}
\label{app:pw_amb}

\begin{proof}[Proof of Lemma \ref{lem:pw_amb}]
It follows from Plancherel's theorem that the restriction of $f$ to the real line is in
$L^2(\bbR)$. Therefore, $\calA f (x,\omega) = \rme^{\pi \rmi x \omega} \calV_f f(x,\omega)$ is
uniformly continuous \cite{Groechenig01}. Now, define $f_x(t) := \overline{f(t-x)}$, for $t,x\in
\bbR$, and let $x \in \bbR$ be arbitrary but fixed. We compute
\[
    \rme^{-\pi\rmi x\omega} \calA f (x,\omega) = \int_\bbR f(t) \overline{f(t-x)} \rme^{-2\pi\rmi t\omega} \dd t
    = \calF\left( f \cdot f_x \right)(\omega)
    = \left( \calF f \ast \calF f_x \right)(\omega),
\]
using the convolution theorem. Furthermore, we find that
\begin{align*}
    \calF f_x(\omega) &= \int_\bbR \overline{f(t-x)} \rme^{-2\pi\rmi t\omega} \dd t
    = \rme^{-2\pi\rmi x\omega} \int_\bbR \overline{f(t)} \rme^{2\pi\rmi t (-\omega)} \dd t
    = \rme^{-2\pi\rmi x\omega} \overline{\calF f(-\omega)} \\ &= \rme^{-2\pi\rmi x\omega} \overline{\calF f(-\omega)}.
\end{align*}
Therefore, we have 
\[
    \rme^{-\pi\rmi x\omega} \calA f (x,\omega) = \int_\bbR \calF f(\xi) \calF f_x(\omega - \xi) \dd \xi
    = \int_{-B}^B \calF f(\xi) \overline{\calF f(\xi-\omega)} \rme^{-2\pi\rmi x(\omega-\xi)} \dd \xi.
\]
If $\omega \in \bbR$ is such that $\abs{\omega} \geq 2B$, then we can readily see that $\xi - \omega
\not\in [-B,B]$, for $\xi \in (-B,B)$. It follows that $\calA f (x,\omega) = 0$.
\end{proof}

\end{document}